\documentclass[12pt,a4paper]{article}
\usepackage{latexsym,amssymb,amsfonts,amsmath,amsthm,nccmath,float,enumitem,setspace,bm,authblk}
\usepackage[usenames,dvipsnames]{xcolor}
\usepackage[hidelinks]{hyperref}
\usepackage[margin=2cm]{geometry}

\setlist{topsep=3pt,partopsep=0pt,itemsep=1pt,parsep=0pt}

\hypersetup{
    colorlinks,
    linkcolor={red!80!black},
    citecolor={blue!80!black},
    urlcolor={blue!80!black}
}

\newtheorem{Theorem}{Theorem}

\newtheorem{Lemma}{Lemma}

\newcommand{\Z}{\mathbb{Z}}

\newcommand{\B}{\mathcal{B}}
\newcommand{\F}{\mathcal{F}}

\def \leq {\leqslant}
\def \geq {\geqslant}

\let\oldproofname=\proofname
\renewcommand{\proofname}{\rm\bf{\oldproofname}}

\begin{document}

\title{The existence of cyclic $(v,4,1)$-designs}

\author[a]{Menglong Zhang}
\author[a]{Tao Feng}
\author[b]{Xiaomiao Wang}
\affil[a]{School of Mathematics and Statistics, Beijing Jiaotong University, Beijing 100044, P.R. China}
\affil[b]{School of Mathematics and Statistics, Ningbo University, Ningbo 315211, P.R. China}
\affil[ ]{mlzhang@bjtu.edu.cn; tfeng@bjtu.edu.cn; wangxiaomiao@nbu.edu.cn}

\renewcommand*{\Affilfont}{\small\it}
\renewcommand\Authands{ and }
\date{}

\maketitle

\footnotetext{Supported by NSFC under Grant 11871095 (Tao Feng), NSFC under Grant 11771227 and Zhejiang Provincial Natural Science Foundation of China under Grant LY21A010005 (Xiaomiao Wang).}

\begin{abstract}
Even though Peltesohn proved that a cyclic $(v,3,1)$-design exists if and only if $v\equiv 1,3\pmod{6}$ as early as 1939, the problem of determining the spectrum of cyclic $(v,k,1)$-designs with $k>3$ is far from being settled, even for $k=4$. This paper shows that a cyclic $(v,4,1)$-design exists if and only if $v\equiv 1,4\pmod{12}$ and $v\not\in\{16,25,28\}$.
\end{abstract}

\noindent {\bf Keywords}: cyclic design; difference family; optical orthogonal code

\section{Introduction}

Let $X$ be a set of $v$ {\em points}, and $\B$ be a collection of $k$-subsets
of $X$ called {\em blocks}. A pair $(X,\B)$ is called a {\em $(v,k,1)$-design} if every pair of distinct elements of $X$ is contained in exactly one block of $\B$. Kirkman \cite{Kirkman} in 1847 showed that a $(v,3,1)$-design exists if and only if $v\equiv 1,3\pmod{6}$, and Hanani \cite{Hanani} in 1961 showed that a $(v,4,1)$-design exists if and only if $v\equiv 1,4\pmod{12}$.

An {\em automorphism} of a $(v,k,1)$-design $(X, \B)$ is a permutation on $X$ leaving $\B$ invariant. A $(v,k,1)$-design is said to be {\em cyclic} if it admits an automorphism consisting of a cycle of length $v$. Without loss of generality we identify $X$ with ${\mathbb Z}_v$, the additive group of integers modulo $v$. The blocks of a cyclic $(v,k,1)$-design can be partitioned into orbits under ${\mathbb Z}_{v}$. We can choose any fixed block from each orbit and then call these {\em base blocks}. If the cardinality of an orbit is equal to $v$, the orbit is {\em full}. Otherwise, it is {\em short}. If $\gcd(v,k)=1$, then all orbits of a cyclic $(v,k,1)$-design are full (see \cite[Lemma 1]{k}).

The existence problem for cyclic $(v,3,1)$-designs is equivalent to Heffter's difference problems. To generalize one of Netto's constructions \cite{Netto} in 1893 for cyclic $(v,3,1)$-designs, Heffter \cite{Heffter1896} in 1896 introduced his famous first difference problem that is related to constructions for cyclic $(v,3,1)$-designs with $v\equiv 1\pmod{6}$, and a year later both the first and second difference problems appeared \cite{Heffter1897}. Heffter's difference problems were eventually solved in 1939 by Peltesohn \cite{Peltesohn}.

\begin{Theorem}\label{thm:CDF-CD}{\rm \cite{Peltesohn}}
There exists a cyclic $(v,3,1)$-design if and only if $v\equiv 1,3\pmod{6}$ and $v\neq 9$.
\end{Theorem}

The problem of determining the spectrum of cyclic $(v,k,1)$-designs with $k>3$ is far from being settled, even for $k=4$. For small orders, no cyclic $(v,4,1)$-design exists for $v=16,25,28$ \cite{cm}; a cyclic $(12t+1,4,1)$-design exists for any $t\leq 1000$ except for $t=2$ \cite{gsm}; a cyclic $(12t+4,4,1)$-design exists for any $3\leq t\leq 50$ \cite{cw}. It has been conjectured that cyclic $(v,4,1)$-designs exist for all $v\equiv 1,4\pmod{12}$ and $v\geq 37$ \cite{rr}.

When $p\equiv 1\pmod{12}$ is a prime, Bose \cite{Bose} provided a sufficient condition for the existence of a cyclic $(p,4,1)$-design admitting a multiplier of order 3, and the necessary and sufficient condition for this special kind of cyclic $(p,4,1)$-designs was established by Buratti in \cite{Buratti95}. Following Buratti's work in \cite{Buratti95-1}, Chen and Zhu \cite{chen45} showed that a cyclic $(p,4,1)$-design exists for any prime $p\equiv 1\pmod{12}$. When $p\equiv 13\pmod{24}$ is a prime, Check and Colbourn \cite{cc} gave a direct construction for cyclic $(4p^n,4,1)$-designs with any given nonnegative integer $n$. Buratti \cite{Buratti02} presented an explicit construction for cyclic $(4p,4,1)$-designs for any prime $p\equiv 1\pmod{12}$. Actually one can extend any cyclic $(p,k,1)$-design with $p$ a prime to a cyclic $(kp,k,1)$-design (cf. \cite{Buratti97}). On the other hand, by means of recursive constructions (cf. \cite{Buratti98,cc84}) cyclic designs of composite order $v = v_1v_2$ can be obtained. On the whole, it appears that no infinite family of cyclic $(v,4,1)$-design was known such that $v$ can run over a congruent class and $v$ is not a prime. For more information on cyclic $(v,4,1)$-designs, the reader is referred to \cite{AbelBuratti2004,abel,bp2,c}.

As the main result of the paper, we are to prove the following theorem.

\begin{Theorem}\label{thm:main}
There exists a cyclic $(v,4,1)$-design if and only if $v\equiv 1,4\pmod{12}$ and $v\not\in\{16,25,28\}$.
\end{Theorem}

Cyclic designs are closely related to optical orthogonal codes that are widely used as spreading codes in optical code-division multiple access systems \cite{csw}. As a corollary of Theorem \ref{thm:main}, we obtain the following optimal optical orthogonal codes.

\begin{Theorem}\label{thm:main-OOC}
There exists an optimal $(v,4,1)$-optical orthogonal code for any $v\equiv 1,4\pmod{12}$ and $v\not\in\{16,25,28\}$. 
\end{Theorem}

\section{Preliminaries}\label{sec-2}

A useful tool for generating cyclic designs is the concept of cyclic difference families. Every union in this paper will be understood as {\em multiset union}. A {\em $(v,k,1)$-cyclic difference packing} (briefly CDP) is a family $\F$ of $k$-subsets (called {\em base blocks}) of ${\mathbb Z}_{v}$ such that the multiset
$$\Delta \F:=\bigcup_{F\in\F}\Delta F:=\{x-y \pmod{v}:x,y\in F, x\not=y,F\in \F\}$$
contains every element of ${\mathbb Z}_{v}\setminus \{0\}$ at most once. Write $L:={\mathbb Z}_{v}\setminus \Delta \F$, and $L$ is said to be the \emph{difference leave} or {\em leave} of $\F$. If $L=\{0\}$, $\F$ is called a {\em $(v,k,1)$-cyclic difference family}. If $k$ is a divisor of $v$ and $L$ is the subgroup of order $k$ in $\Z_v$, $\F$ is called a {\em $(v,k,k,1)$-cyclic difference family} (briefly CDF). A $(v,k,1)$-CDF contains $(v-1)/k(k-1)$ base blocks, and a $(v,k,k,1)$-CDF contains $(v-k)/k(k-1)$ base blocks.

\begin{Lemma}\label{lem:CDF-CD}{\rm \cite{cm}}
\begin{itemize}
\item [$(1)$] Let $\F$ be a $(v,k,1)$-CDF. Then $\F$ forms the set of base blocks of a cyclic $(v,k,1)$-design.
\item [$(2)$] Let $\F$ be a $(v,k,k,1)$-CDF. Then $\F\cup \{\{0,v/k,2v/k,\ldots,(k-1)v/k\}\}$ forms the set of base blocks of a cyclic $(v,k,1)$-design.
\end{itemize}
\end{Lemma}

For any base block $F$ of a $(v,k,1)$-CDP $\F$, if $x,y\in F$ and $x>y$, we call $x-y$ a {\em positive difference} from $F$, and $y-x \pmod{v}$ a {\em negative difference} from $F$. The collection of all positive differences (resp. negative differences) in $\Delta F$ is denoted by $\Delta^+ F$ (resp. $\Delta^- F$). Write $\Delta^+\F=\bigcup_{F\in\F} \Delta^+ F$ and $\Delta^-\F=\bigcup_{F\in\F} \Delta^- F$. Clearly $\Delta \F=$ $\Delta^+ \F\cup \Delta^- \F$.

For positive integers $a,b$ and $c$ such that $a\leq b$ and $a\equiv b\pmod{c}$, we set $[a,b]_c:=\{a+ci:0\leq i\leq (b-a)/c\}$. When $c=1$, $[a,b]_1$ is simply written as $[a,b]$.

\section{Direct constructions for cyclic difference families}

The idea that we use to construct a $(v,4,1)$-CDF is from \cite{yl} which is a monograph on coding theory and cryptography written by Yang and Lin in Chinese in 1992. In Section \ref{sec:revisit} we give a review of Yang and Lin's construction for optical orthogonal codes, which are equivalent to cyclic difference packings. By modifying Yang and Lin's construction slightly, we obtain a $(v,4,1)$-CDF for any $v\equiv 1\pmod{72}$ in Section \ref{sec:slight}. Further modification of Yang and Lin's construction is made in Section \ref{sec:further} to produce $(v,4,1)$-CDFs and $(v,4,4,1)$-CDFs for all admissible values of $v$.

\subsection{Revisit of Yang and Lin's construction}\label{sec:revisit}

A $(v,k,1)$-{\em optical orthogonal code} (briefly OOC) $\cal{C}$, is a family of $(0,1)$ sequences (called {\em codewords}) of length $v$ and weight $k$ satisfying that
for any ${\mathbf{x}}=(x_0,x_1,\ldots,x_{v-1})\in\cal{C}$, ${\mathbf{y}}=(y_0,y_1,\ldots,y_{v-1})\in\cal{C}$ and any integer $r$,
$\sum_{i=0}^{v-1}x_iy_{i+r}\leq 1,$
where either ${\mathbf{x}}\neq {\mathbf{y}}$ or $r\neq 0$, and the arithmetic $i+r$ is reduced modulo $v$. A $(v,k,1)$-OOC with $\lfloor (v-1)/k(k-1)\rfloor$ codewords is said to be {\em optimal}.

\begin{Lemma}\label{lem:CDP-OOC} {\rm \cite[Theorem 2.1]{Yin98}}
A $(v,k,1)$-CDP with $b$ base blocks is equivalent to a $(v,k,1)$-OOC with $b$ codewords.
\end{Lemma}

Yang and Lin \cite{yl} constructed a $(v,4,1)$-OOC with $\frac{v-1}{12}-2$ codewords for any $v\equiv 1 \pmod{72}$. Actually one of these codewords is not correct, so their $(v,4,1)$-OOC consists of $\frac{v-1}{12}-3$ codewords. Lemma \ref{lem:CDP-OOC} establishes the equivalence between $(v,4,1)$-OOCs and $(v,4,1)$-CDPs. Thus modifying Yang and Lin's construction to obtain $(v,4,1)$-CDFs is worth pursuing. We include Yang and Lin's construction here to facilitate the reader to compare their construction with ours.

\begin{Lemma}{\rm \cite{yl}} \label{lem:Yang}
There exists a $(v,4,1)$-CDP with $\frac{v-1}{12}-3$ base blocks for any integer $v\equiv 1\pmod{72}$ and $v>1$.
\end{Lemma}

\begin{proof} Let $v=72t+1$ and $t>0$. The $6t-3$ base blocks are listed below:
\begin{center}
\begin{tabular}{lllll}
$F_{1,i}=\{0$, & $43t+i$,   & $31t+1+2i$,   & $8t+2+3i\}$, & $i\in I_1$; \\
$F_{2,i}=\{0$, & $23t+i$,   & $5t+1+2i$,    & $8t+1+3i\}$, & $i\in I_2$; \\
$F_{3,i}=\{0$, & $41t+i$,   & $25t+2i$,     & $8t+3i\}$, &  $i\in I_3$; \\
$F_{4,i}=\{0$, & $35t+i$,   & $5t+2i$,      & $1+3i\}$, &  $i\in I_4$; \\
$F_{5,i}=\{0$, & $47t+2+i$, & $19t+1+2i$,   & $2+3i\}$, &  $i\in I_5$; \\
$F_{6,i}=\{0$, & $21t+i$,   & $13t+2i$,     & $3i\}$, &  $i\in I_6$, \\
\end{tabular}
\end{center}
where $I_1=I_3=I_6=\{i:1\leq i\leq t-1\}$ and $I_2=I_4=I_5=\{i:0\leq i\leq t-1\}$. \newpage

There is no detailed explanation in \cite{yl} on how Yang and Lin found the above base blocks, and the verification of the correctness of these base blocks is left to the reader. In order to obtain some intuition on the choices of these base blocks, we provide details here to check their construction.
Let ${\cal F}_{r}=\{F_{r,i}:i\in I_r\}$ for $1\leq r\leq 6$. All positive differences from ${\cal F}_r$ are listed in Table \ref{tab0}, and if a positive difference is greater than $36t$, we list its corresponding negative difference. Let
\begin{align*}
L= &  \{0\}\cup \pm\{7t,8t,12t,15t,19t,27t\}
\\ & \ \ \ \ \cup \pm\{8t+2,12t-1,13t,16t,21t,23t-1,25t,27t+1,30t+1,31t+1,35t-2,36t\}.
\end{align*}
It is readily checked that $\Delta(\bigcup_{r=1}^6 \F_r)$ covers every element in ${\mathbb Z}_{v}\setminus L$ exactly once.
\end{proof}

\begin{table}[t]
\caption{Differences from base blocks in Lemma \ref{lem:Yang}} \label{tab0}
\begin{center}
\begin{tabular}{|c|c|c||c|c|}\hline
$\Delta F^+_{1,i}$ & $\Delta{\cal F}^+_1$ & $\Delta{\cal F}^-_1$ & $\Delta F^+_{2,i}$ & $\Delta{\cal F}^+_2$     \\\hline
$43t+i$ & $[43t+1,44t-1]$ & $[28t+2,29t]$ & $23t+i$ & $[23t,24t-1]$     \\
$31t+1+2i$ & $[31t+3,33t-1]_{2}$ & & $5t+1+2i$ & $[5t+1,7t-1]_{2}$  \\
$8t+2+3i$ & $[8t+5,11t-1]_{3}$ & & $8t+1+3i$ & $[8t+1,11t-2]_{3}$  \\
$12t-1-i$ & $[11t,12t-2]$ & & $18t-1-i$ & $[17t,18t-1]$  \\
$35t-2-2i$ & $[33t,35t-4]_{2}$ & & $15t-1-2i$ & $[13t+1,15t-1]_{2}$  \\
$23t-1-i$ & $[22t,23t-2]$ & & $3t+i$ & $[3t,4t-1]$  \\\hline
$\Delta F^+_{3,i}$ & $\Delta{\cal F}^+_3$ & $\Delta{\cal F}^-_3$ & $\Delta F^+_{4,i}$ & $\Delta{\cal F}^+_4$     \\\hline
$41t+i$ & $[41t+1,42t-1]$ & $[30t+2,31t]$ & $35t+i$ & $[35t,36t-1]$     \\
$25t+2i$ & $[25t+2,27t-2]_{2}$ & & $5t+2i$ & $[5t,7t-2]_{2}$  \\
$8t+3i$ & $[8t+3,11t-3]_{3}$ & & $1+3i$ & $[1,3t-2]_{3}$  \\
$16t-i$ & $[15t+1,16t-1]$ & & $30t-i$ & $[29t+1,30t]$  \\
$33t-2i$ & $[31t+2,33t-2]_{2}$ & & $35t-1-2i$ & $[33t+1,35t-1]_{2}$  \\
$17t-i$ & $[16t+1,17t-1]$ & & $5t-1-i$ & $[4t,5t-1]$  \\\hline
$\Delta F^+_{5,i}$ & $\Delta{\cal F}^+_5$ & $\Delta{\cal F}^-_5$ & $\Delta F^+_{6,i}$ & $\Delta{\cal F}^+_6$     \\\hline
$47t+2+i$ & $[47t+2,48t+1]$ & $[24t,25t-1]$ & $21t+i$ & $[21t+1,22t-1]$     \\
$19t+1+2i$ & $[19t+1,21t-1]_{2}$ & & $13t+2i$ & $[13t+2,15t-2]_{2}$  \\
$2+3i$ & $[2,3t-1]_{3}$ & & $3i$ & $[3,3t-3]_{3}$  \\
$28t+1-i$ & $[27t+2,28t+1]$ & & $8t-i$ & $[7t+1,8t-1]$  \\
$47t-2i$ & $[45t+2,47t]_{2}$ & $[25t+1,27t-1]_2$ & $21t-2i$ & $[19t+2,21t-2]_{2}$  \\
$19t-1-i$ & $[18t,19t-1]$ & & $13t-i$ & $[12t+1,13t-1]$  \\\hline
\end{tabular}
\end{center}
\end{table}

\subsection{Slight modification of Yang and Lin's construction}\label{sec:slight}

In Yang and Lin's construction, all base blocks are divided into 6 parts. Each base block in the $r$-th part, $1\leq r\leq 6$, is of the form
\begin{align}\label{eqn:1}
\{0,\ \ \alpha_{r1}t+\alpha_{r2}+i,\ \ \beta_{r1}t+\beta_{r2}+2i,\ \ \gamma_{r1}t+\gamma_{r2}+3i\},
\end{align}
where $i$ runs over some set $I_r$. By choosing appropriate parameters $\alpha_{r1},\alpha_{r2},\beta_{r1},\beta_{r2},\gamma_{r1},\gamma_{r2}$ and $I_r$, Yang and Lin constructed $(v,4,1)$-CDPs shown in Lemma \ref{lem:Yang}.

Observe the difference leave $L$ of the $(v,4,1)$-CDP in the proof of Lemma \ref{lem:Yang}. Each difference in $L$ is almost a multiple of $t$. It is easy to see that we can add one more base block, $\{0,7t,19t,64t+1\}$, to form a $(v,4,1)$-CDP with $\frac{v-1}{12}-2$ base blocks. However, we cannot extend it anymore. A natural idea to solve this problem is to reduce the range of values for $I_r$ such that more differences are released and then reassemble them to produce a CDF. In our construction (see Lemma \ref{thm:CDF1mod72}), we set $I_r=\{i:1\leq i\leq t-2\}\setminus\{\lfloor t/2\rfloor\}$ for each $1\leq r\leq 6$. Note that the lack of $\lfloor t/2\rfloor$ in our $I_r$ ensures that each difference in $L$ is around some multiple of $\lfloor t/2\rfloor$. This increases the flexibility to complete a CDP to a CDF.

\begin{Lemma}\label{thm:CDF1mod72}
There exists a $(v,4,1)$-CDF for any positive integer $v\equiv 1\pmod{72}$.
\end{Lemma}

\begin{proof} For $v\in\{73,145\}$, a $(v,4,1)$-CDF exists by \cite[Theorem 16.28]{abel}. For $v\equiv 1\pmod{72}$ and $v>145$, let $v=72t+1$ where $t>2$. A $(v,4,1)$-CDF, $\F$, contains $6t$ base blocks. The first $6t-18$ base blocks are listed below:
\begin{center}
\begin{tabular}{llll}
$F_{1,i}:=\{0$, & $43t+i$, & $31t+1+2i$, & $8t+2+3i\}$, \\
$F_{2,i}:=\{0$, & $23t+i$, & $5t+1+2i$, & $8t+1+3i\}$, \\
$F_{3,i}:=\{0$, & $41t+i$, & $25t+2i$, & $8t+3i\}$, \\
$F_{4,i}:=\{0$, & $35t+i$, & $5t+2i$, & $1+3i\}$, \\
$F_{5,i}:=\{0$, & $47t+2+i$, & $19t+1+2i$, & $2+3i\}$, \\
$F_{6,i}:=\{0$, & $21t+i$, & $13t+2i$, & $3i\}$, \\
\end{tabular}
\end{center}
where $1\leq i\leq t-2$ and $i\neq \lfloor t/2\rfloor$. The remaining 18 base blocks are given according to the parity of $t$. If $t$ is odd, we take
\begin{center}
\begin{tabular}{lll}
$\{0,1,\frac{3t-1}{2},11t-2\}$,&$\{0,2,3t-1,15t\}$,&
$\{0,3t-2,11t-1,35t-2\}$,\\$\{0,4t-1,12t-1,27t-2\}$,&
$\{0,\frac{3t+1}{2},\frac{57t+3}{2},\frac{71t-1}{2}\}$,&$\{0,5t-1,13t+1,34t\}$,\\
$\{0,7t+1,28t+1,42t\}$,&$\{0,17t,36t+2,\frac{95t+3}{2}\}$,&
$\{0,\frac{37t-1}{2},\frac{43t-1}{2},\frac{89t-1}{2}\}$,\\$\{0,7t-1,18t-1,42t-1\}$,&$\{0,4t,\frac{33t+1}{2},40t\}$,&

$\{0,6t-1,33t+1,45t+1\}$,\\
$\{0,15t+1,31t+2,56t+1\}$,&$\{0,7t,25t,53t+2\}$,&
$\{0,\frac{15t+1}{2},\frac{59t+1}{2},\frac{99t+3}{2}\}$,\\
$\{0,\frac{9t-1}{2},\frac{19t-1}{2},\frac{109t+3}{2}\}$,&
$\{0,6t,\frac{31t+1}{2},19t\}$,&$\{0,5t+1,25t+1,58t+1\}$.
\end{tabular}
\end{center}
If $t$ is even, we take
\begin{center}
\begin{tabular}{lll}
$\{0,1,3t-2,11t-2\}$,&$\{0,3t-1,15t-2,60t+1\}$,&
$\{0,\frac{3t}{2}+2,6t+1,19t+1\}$,\\$\{0,2,35t,42t+1\}$,&
$\{0,8t+2,38t+3,\frac{99t}{2}+2\}$,&$\{0,\frac{19t}{2}+2,\frac{57t}{2}+1,\frac{107t}{2}+2\}$,\\
$\{0,\frac{19t}{2},\frac{33t}{2},\frac{95t}{2}+2\}$,&$\{0,20t,39t+2,46t\}$,&
$\{0,\frac{59t}{2},31t+1,45t\}$,\\$\{0,\frac{3t}{2},25t,64t\}$,&
$\{0,15t+1,20t+1,\frac{55t}{2}+1\}$,&$\{0,\frac{7t}{2},40t+1,54t+1\}$,\\
$\{0,4t,15t,51t+2\}$,&$\{0,5t-1,22t-1,50t+1\}$,&
$\{0,6t,24t-1,51t+1\}$,\\$\{0,3t,16t+1,39t+1\}$,&
$\{0,12t+1,16t,65t+2\}$,&$\{0,5t+1,29t+1,40t\}$.\\
\end{tabular}
\end{center}
One can make a table similar to Table \ref{tab0} for the first $6t-18$ base blocks and then check that $\Delta \F=\Z_v\setminus\{0\}$. Thus $\F$ is a $(v,4,1)$-CDF. \end{proof}

To facilitate the reader to check the correctness of our results, we provide a computer code written by GAP \cite{GAP4} to show that our constructions in Lemma \ref{thm:CDF1mod72}, Lemma \ref{thm:CDF13mod72} and Theorem \ref{thm:CDF4mod12} always work regardless of the parameter $t$. The interested reader can get a copy of the computer code from \cite{zfw}.

\subsection{Further modification of Yang and Lin's construction}\label{sec:further}

To construct $(v,4,1)$-CDFs with $v\not\equiv 1\pmod{72}$, we need to modify the values of $\alpha_{r2}$, $\beta_{r2}$ and $\gamma_{r2}$ in \eqref{eqn:1}.

\begin{Lemma}\label{thm:CDF13mod72}
There exists a $(v,4,1)$-CDF for any positive integer $v\equiv 13,25,37,49,61\pmod{72}$ and $v\neq 25$.
\end{Lemma}

\begin{proof} For $v\equiv 13,25,37,49,61\pmod{72}$, $v\neq 25$ and $v\leq 205$, a $(v,4,1)$-CDF exists by \cite[Theorem 16.28]{abel}. For $v>205$, let $v=72t+12x+1$ where $t>2$ and $1\leq x\leq 5$. A $(v,4,1)$-CDF contains $6t+x$ base blocks. The first $6t-18$ base blocks are listed below:
\begin{center}
\begin{tabular}{llll}
$\{0$, & $43t+a_1+i$, & $31t+a_2+2i$, & $8t+a_3+3i\}$,\\
$\{0$, & $23t+b_1+i$, & $5t+b_2+2i$, & $8t+b_3+3i\}$,\\
$\{0$, & $41t+c_1+i$, & $25t+c_2+2i$, & $8t+c_3+3i\}$,\\
$\{0$, & $35t+d_1+i$, & $5t+d_2+2i$, & $d_3+3i\}$,\\
$\{0$, & $47t+e_1+i$, & $19t+e_2+2i$, & $e_3+3i\}$,\\
$\{0$, & $21t+f_1+i$, & $13t+f_2+2i$, & $f_3+3i\}$,\\
\end{tabular}
\end{center}
where $1\leq i\leq t-2$, $i\neq \lfloor t/2\rfloor$, and $a_j,b_j,c_j,d_j,e_j,f_j$ for $1\leq j\leq 3$ are given in the following table:
\begin{center}
\begin{tabular}{|l|lll|lll|lll|lll|lll|lll|}\hline
$x$ & $a_1$ & $a_2$ & $a_3$ & $b_1$ & $b_2$ & $b_3$ & $c_1$ & $c_2$ & $c_3$ & $d_1$ & $d_2$ & $d_3$ & $e_1$ & $e_2$ & $e_3$ & $f_1$ & $f_2$ & $f_3$\\\hline
$1$ & $8$ & $7$ & $5$ & $4$ & $3$ & $4$ & $9$ & $6$ & $3$ & $5$ & $2$ & $1$ & $10$ & $3$ & $2$ & $2$ & $1$ & $0$ \\
$2$ & $16$ & $14$ & $8$ & $10$ & $4$ & $7$ & $15$ & $11$ & $6$ & $12$ & $3$ & $1$ & $17$ & $9$ & $2$ & $6$ & $4$ & $0$ \\
$3$ & $25$ & $20$ & $10$ & $14$ & $5$ & $8$ & $24$ & $18$ & $9$ & $17$ & $4$ & $1$ & $24$ & $11$ & $2$ & $12$ & $7$ & $0$ \\
$4$ & $31$ & $24$ & $10$ & $18$ & $7$ & $8$ & $30$ & $20$ & $9$ & $23$ & $6$ & $1$ & $32$ & $13$ & $2$ & $16$ & $11$ & $0$ \\
$5$ & $38$ & $30$ & $11$ & $21$ & $8$ & $9$ & $37$ & $24$ & $10$ & $29$ & $7$ & $1$ & $40$ & $17$ & $2$ & $18$ & $13$ & $0$ \\\hline
 \end{tabular} .
\end{center}
The remaining $18+x$ base blocks are provided in Table \ref{tab1} according to the parity of $t$. \end{proof}

In the proof of Lemma \ref{thm:CDF13mod72}, the first $6t-18$ base blocks were found by hand, and the latter $18+x$ base blocks were found by computer search. We illustrate why and how to modify the values of $\alpha_{r2}$, $\beta_{r2}$ and $\gamma_{r2}$ in \eqref{eqn:1} to get the $6t-18$ base blocks in Lemma \ref{thm:CDF13mod72}. For instance, when $v\equiv 13\pmod{72}$, let $v=72t+13$. Consider Yang and Lin's base blocks in Lemma \ref{lem:Yang}, which can be rewritten as follows:
\begin{center}
\begin{tabular}{lllll}
$F_{1,i}=\{0$, & $43t+a_1+i$, & $31t+a_2+2i$, & $8t+a_3+3i\}$, & $i\in I_1$; \\
$F_{2,i}=\{0$, & $23t+b_1+i$, & $5t+b_2+2i$, & $8t+b_3+3i\}$, & $i\in I_2$; \\
$F_{3,i}=\{0$, & $41t+c_1+i$, & $25t+c_2+2i$, & $8t+c_3+3i\}$, &  $i\in I_3$; \\
$F_{4,i}=\{0$, & $35t+d_1+i$, & $5t+d_2+2i$, & $d_3+3i\}$, &  $i\in I_4$; \\
$F_{5,i}=\{0$, & $47t+e_1+i$, & $19t+e_2+2i$, & $e_3+3i\}$, &  $i\in I_5$; \\
$F_{6,i}=\{0$, & $21t+f_1+i$, & $13t+f_2+2i$, & $f_3+3i\}$, &  $i\in I_6$, \\
\end{tabular}
\end{center}
where $I_1=I_3=I_6=\{i:1\leq i\leq t-1\}$, $I_2=I_4=I_5=\{i:0\leq i\leq t-1\}$, $(a_1,a_2,a_3)=(0,1,2)$, $(b_1,b_2,b_3)=(0,1,1)$, $(c_1,c_2,c_3)=(0,0,0)$, $(d_1,d_2,d_3)=(0,0,1)$, $(e_1,e_2,e_3)=(2,1,2)$ and $(f_1,f_2,f_3)=(0,0,0)$.

Let's analyze why some of Yang and Lin's base blocks are not valid for $v\equiv 13\pmod{72}$. By Table \ref{tab0}, $[43t+1,44t-1]\in \Delta{\cal F}^+_1$. Then $v=72t+13$ implies that $\Delta{\cal F}^-_1$ contains the set $[28t+14,29t+12]$, which interacts with the set $[29t+1,30t]\in \Delta{\cal F}^+_4$ on $12$ elements. With the same argument, $[41t+1,42t-1]\in \Delta{\cal F}^+_3$ leads to $[30t+14,31t+12]\in \Delta{\cal F}^-_3$, which intersects with both the sets $[31t+3,33t-1]_2\in \Delta{\cal F}^+_1$ and $[31t+2,33t-2]_{2}\in \Delta{\cal F}^+_3$. Also, the sets $[47t+2,48t+1]$ and $[45t+2,47t]_{2}$ from $\Delta{\cal F}^+_5$ lead to $[24t+12,25t+11], [25t+13,27t+11]_{2}\in\Delta{\cal F}^-_5$, which intersect with $[25t+2,27t-2]_{2}\in \Delta{\cal F}^+_3$ and $[27t+2,28t+1]\in \Delta{\cal F}^+_5$ respectively.

On one hand, by the above observation, we could delete all the base blocks that produce overlapped differences and then reassemble them to form a $(v,4,1)$-CDF with $v\equiv 13\pmod{72}$. However, this would lead to a big leave that contains too many differences to extend the resulting CDP to a CDF by computer search. On the other hand, we can try to make slight adjustment for the values of $a_1, a_2, a_3,\ldots, f_1, f_2, f_3$ such that the intervals of differences are mutually disjoint and the leave of the CDP is as small as possible. We did this procedure by hand and it often spent us several hours for each $v$ modulo $72$. Once a $(v,4,1)$-CDF with $v\equiv 13\pmod{72}$ was found, we took its first $6t-18$ base blocks, and then did the same strategy for $v\equiv 25\pmod{72}$ to get new $6t-18$ base blocks, and so on. For each case of $v$, the time used to search for the remaining $18+x$ base blocks by a common personal computer was ranged from a half day to two days.

\begin{table}[htbp] \renewcommand\arraystretch{1.09}
\caption{The remaining $18+x$ base blocks in Lemma \ref{thm:CDF13mod72}}\label{tab1}
\begin{center}\tabcolsep 0.03in
\begin{tabular}{|lll|}\hline
$x=1$ and $t$ is odd &&\\\hline
$\{0,2,35t+5,41t+7\}$&$\{0,\frac{3t-3}{2},14t,56t+10\}$&
$\{0,\frac{3t-1}{2},13t+1,25t+3\}$\\
$\{0,\frac{3t+1}{2},5t+1,23t+2\}$&
$\{0,3t-3,3t-2,35t+4\}$&$\{0,3t,27t+4,68t+13\}$\\
$\{0,4t+2,27t+6,68t+12\}$&$\{0,5t+2,20t+2,\frac{135t+23}{2}\}$&
$\{0,6t+1,25t+5,52t+10\}$\\
$\{0,\frac{19t+3}{2},26t+5,45t+6\}$&
$\{0,11t,22t+2,\frac{129t+23}{2}\}$&$\{0,12t+1,19t+2,27t+3\}$\\
$\{0,15t-1,23t+3,65t+11\}$&$\{0,15t+1,34t+4,42t+9\}$&
$\{0,\frac{31t+7}{2},25t+6,59t+11\}$ \\
$\{0,18t+2,21t+1,57t+10\}$&
$\{0,21t+2,24t+3,57t+9\}$&$\{0,\frac{45t+5}{2},\frac{61t+9}{2},\frac{89t+11}{2}\}$\\
$\{0,25t+4,36t+5,41t+8\}$&&\\\hline
$x=1$ and $t$ is even &&\\\hline
$\{0,3t-2,15t,40t+6\}$&$\{0,\frac{3t}{2}+2,19t+3,26t+5\}$&
$\{0,3t-3,40t+7,\frac{125t}{2}+9\}$\\$\{0,\frac{3t}{2},23t+2,64t+8\}$&
$\{0,3t,22t+2,60t+12\}$&$\{0,3t+1,45t+9,51t+11\}$\\
$\{0,4t+1,8t+1,58t+13\}$&$\{0,\frac{7t}{2}+1,5t+2,27t+5\}$&
$\{0,\frac{9t}{2}+1,\frac{71t}{2}+5,\frac{87t}{2}+8\}$\\$\{0,5t+3,25t+5,61t+12\}$&
$\{0,8t+2,8t+4,53t+9\}$&$\{0,11t,\frac{37t}{2}+1,34t+4\}$\\
$\{0,\frac{23t}{2}+1,\frac{59t}{2}+3,\frac{89t}{2}+6\}$&$\{0,14t+1,21t+1,\frac{61t}{2}+4\}$&
$\{0,15t-1,30t+3,41t+5\}$\\$\{0,15t+2,20t+3,45t+6\}$&
$\{0,19t+1,25t+4,56t+10\}$&$\{0,23t+3,36t+4,36t+5\}$\\
$\{0,24t+4,27t+3,31t+5\}$ &&\\\hline
$x=2$ and $t$ is odd &&\\\hline
$\{0,\frac{3t-3}{2},\frac{43t+11}{2},\frac{55t+17}{2}\}$&$\{0,3t-3,22t+6,47t+14\}$&
$\{0,3t,14t+3,30t+9\}$\\
$\{0,15t+2,23t+6,64t+17\}$&
$\{0,3t+3,24t+9,57t+22\}$&$\{0,4t+3,19t+7,69t+27\}$\\
$\{0,\frac{15t+5}{2},\frac{31t+9}{2},25t+9\}$&
$\{0,8t+3,39t+14,44t+16\}$&$\{0,29t+10,37t+15,37t+17\}$\\$\{0,\frac{19t+11}{2},\frac{47t+19}{2},\frac{73t+27}{2}\}$&
$\{0,\frac{19t+13}{2},38t+16,\frac{99t+37}{2}\}$&$\{0,12t+2,12t+3,30t+10\}$\\
$\{0,12t+4,37t+14,43t+16\}$&$\{0,15t+5,\frac{33t+11}{2},\frac{137t+45}{2}\}$&
$\{0,3t+2,36t+11,67t+21\}$\\$\{0,17t+5,\frac{59t+19}{2},34t+12\}$&
$\{0,17t+6,40t+15,\frac{83t+29}{2}\}$&$\{0,19t+8,22t+7,35t+12\}$\\
$\{0,23t+8,26t+9,30t+11\}$&$\{0,7t+2,12t+5,56t+20\}$&\\\hline
$x=2$ and $t$ is even &&\\\hline
$\{0,1,15t+4,64t+19\}$&$\{0,2,12t+4,28t+10\}$&
$\{0,\frac{3t}{2},30t+9,64t+17\}$\\$\{0,\frac{3t}{2}+1,17t+5,\frac{37t}{2}+7\}$&
$\{0,3t-1,32t+9,68t+23\}$&$\{0,4t+3,7t+3,25t+9\}$\\
$\{0,\frac{9t}{2}+2,27t+8,64t+23\}$&$\{0,5t+2,24t+9,36t+12\}$&
$\{0,5t+3,18t+8,37t+17\}$\\$\{0,5t+4,33t+13,50t+20\}$&
$\{0,6t+3,\frac{71t}{2}+12,47t+14\}$&$\{0,7t+2,23t+7,26t+10\}$\\
$\{0,\frac{33t}{2}+5,26t+11,\frac{95t}{2}+17\}$&$\{0,8t+3,30t+10,33t+11\}$&
$\{0,8t+4,23t+9,37t+13\}$\\$\{0,12t+5,15t+2,35t+11\}$&
$\{0,14t+3,\frac{35t}{2}+6,\frac{47t}{2}+10\}$&$\{0,\frac{15t}{2}+2,\frac{55t}{2}+8,\frac{61t}{2}+10\}$\\
$\{0,18t+7,31t+11,56t+21\}$&$\{0,22t+6,30t+11,33t+9\}$&\\\hline
\end{tabular}
\end{center}
\end{table}

\setcounter{table}{1}
\begin{table}[htbp]  \renewcommand\arraystretch{1.1}
\caption{ (Cont.) The remaining $18+x$ base blocks in Lemma \ref{thm:CDF13mod72}}\label{tab1-1}
\begin{center}\tabcolsep 0.03in
\begin{tabular}{|lll|}\hline
$x=3$ and $t$ is odd &&\\\hline
$\{0,1,8t+9,55t+28\}$&$\{0,\frac{3t-3}{2},23t+10,64t+27\}$&
$\{0,14t+7,35t+16,47t+24\}$\\$\{0,3t+3,8t+7,57t+30\}$&
$\{0,4t+2,41t+22,60t+32\}$&$\{0,25t+16,32t+19,43t+25\}$\\
$\{0,6t+4,\frac{35t+19}{2},\frac{137t+69}{2}\}$&$\{0,7t+4,36t+17,43t+23\}$&
$\{0,20t+10,24t+14,57t+32\}$\\$\{0,12t+6,15t+6,23t+11\}$&
$\{0,13t+8,16t+6,41t+21\}$&$\{0,12t+7,16t+10,65t+35\}$\\
$\{0,\frac{3t+1}{2},\frac{49t+27}{2},\frac{87t+49}{2}\}$&$\{0,\frac{33t+19}{2},26t+17,\frac{113t+61}{2}\}$&
$\{0,19t+12,19t+14,22t+11\}$\\$\{0,\frac{45t+21}{2},\frac{55t+27}{2},\frac{107t+55}{2}\}$&
$\{0,\frac{15t+11}{2},20t+13,56t+29\}$&$\{0,24t+13,31t+18,45t+24\}$\\
$\{0,\frac{59t+27}{2},31t+13,\frac{71t+33}{2}\}$&$\{0,5t+5,8t+6,38t+20\}$&
$\{0,25t+14,38t+21,41t+23\}$\\\hline
$x=3$ and $t$ is even &&\\\hline
$\{0,2,8t+10,37t+22\}$&$\{0,\frac{3t}{2}+2,\frac{19t}{2}+8,\frac{101t}{2}+25\}$& $\{0,25t+15,25t+16,28t+13\}$\\
$\{0,\frac{33t}{2}+9,26t+18,\frac{113t}{2}+31\}$&$\{0,\frac{7t}{2}+3,\frac{37t}{2}+9,\frac{89t}{2}+24\}$
&$\{0,\frac{47t}{2}+14,41t+23,\frac{85t}{2}+24\}$\\
$\{0,20t+11,24t+14,27t+13\}$&$\{0,11t+5,17t+9,64t+28\}$&
$\{0,\frac{23t}{2}+5,40t+17,\frac{99t}{2}+27\}$\\
$\{0,25t+14,36t+20,41t+24\}$&
$\{0,13t+8,16t+9,49t+24\}$&$\{0,16t+6,29t+13,36t+18\}$\\
$\{0,16t+7,30t+13,35t+16\}$&$\{0,7t+4,14t+7,19t+12\}$&
$\{0,17t+10,21t+12,51t+26\}$\\
$\{0,18t+10,33t+17,49t+25\}$&
$\{0,8t+5,29t+14,33t+18\}$&$\{0,23t+10,31t+17,54t+28\}$\\
$\{0,3t+3,6t+5,\frac{141t}{2}+37\}$& $\{0,7t+6,19t+11,22t+11\}$&
$\{0,12t+6,\frac{49t}{2}+13,60t+30\}$\\\hline
$x=4$ and $t$ is odd \\\hline
$\{0,1,3t-2,11t+8\}$&$\{0,2,4t+2,53t+33\}$&
$\{0,\frac{3t-3}{2},23t+14,64t+40\}$\\$\{0,3t-1,15t+9,39t+26\}$&
$\{0,3t+1,15t+10,31t+20\}$&$\{0,4t+3,23t+15,40t+27\}$\\
$\{0,4t+4,39t+25,51t+33\}$&$\{0,4t+5,20t+17,42t+32\}$&
$\{0,27t+19,\frac{61t+39}{2},46t+30\}$\\$\{0,5t+7,12t+11,69t+49\}$&
$\{0,7t+5,36t+23,41t+28\}$&$\{0,14t+10,37t+26,41t+27\}$\\
$\{0,\frac{15t+11}{2},\frac{57t+37}{2},\frac{121t+83}{2}\}$&$\{0,8t+5,13t+11,60t+42\}$&
$\{0,14t+11,35t+22,39t+28\}$\\$\{0,8t+8,39t+27,55t+38\}$&
$\{0,\frac{19t+17}{2},\frac{55t+39}{2},45t+31\}$&$\{0,11t+5,30t+19,60t+37\}$\\
$\{0,11t+6,\frac{59t+35}{2},\frac{119t+75}{2}\}$&$\{0,\frac{9t+11}{2},6t+5,27t+20\}$&
$\{0,8t+6,\frac{19t+13}{2},\frac{111t+75}{2}\}$\\
$\{0,7t+6,13t+12,31t+24\}$&&\\\hline
$x=4$ and $t$ is even &&\\\hline
$\{0,1,\frac{19t}{2}+9,39t+26\}$&$\{0,\frac{3t}{2}+1,\frac{19t}{2}+10,\frac{87t}{2}+31\}$&
$\{0,\frac{33t}{2}+11,\frac{83t}{2}+30,\frac{89t}{2}+30\}$\\$\{0,3t-2,30t+17,35t+22\}$&
$\{0,3t+1,8t+8,26t+19\}$&$\{0,\frac{7t}{2}+1,\frac{31t}{2}+10,51t+33\}$\\
$\{0,4t+1,16t+11,64t+43\}$&$\{0,4t+4,27t+20,41t+30\}$&
$\{0,11t+8,27t+18,46t+29\}$\\$\{0,6t+6,23t+17,53t+35\}$&
$\{0,7t+4,11t+7,\frac{121t}{2}+42\}$&$\{0,7t+5,34t+22,67t+43\}$\\
$\{0,7t+6,11t+6,24t+18\}$&$\{0,\frac{15t}{2}+5,20t+16,\frac{43t}{2}+16\}$&
$\{0,8t+5,41t+27,49t+34\}$\\$\{0,11t+5,15t+10,36t+23\}$&
$\{0,\frac{9t}{2}+5,6t+7,\frac{47t}{2}+18\}$&$\{0,12t+8,41t+26,54t+37\}$\\
$\{0,15t+9,36t+24,56t+37\}$&$\{0,3t-3,3t-1,22t+15\}$&
$\{0,19t+12,23t+14,31t+24\}$\\$\{0,21t+11,25t+17,57t+38\}$&&\\\hline
$x=5$ and $t$ is odd &&\\\hline
$\{0,1,3t-2,60t+48\}$&$\{0,\frac{3t-1}{2},\frac{19t+17}{2},\frac{111t+93}{2}\}$&
$\{0,25t+21,36t+28,41t+36\}$\\$\{0,3t+1,15t+12,39t+33\}$&
$\{0,\frac{7t+1}{2},8t+7,11t+6\}$&$\{0,16t+13,33t+27,40t+33\}$\\
$\{0,4t+1,12t+12,47t+39\}$&$\{0,4t+2,35t+28,54t+46\}$&
$\{0,4t+5,8t+8,72t+59\}$\\$\{0,4t+6,31t+28,53t+46\}$&
$\{0,5t+6,35t+29,53t+45\}$&$\{0,6t+6,31t+30,\frac{85t+77}{2}\}$\\
$\{0,6t+7,\frac{47t+41}{2},51t+44\}$&$\{0,8t+5,21t+18,41t+34\}$&
$\{0,8t+6,30t+25,44t+37\}$\\$\{0,\frac{19t+15}{2},11t+8,57t+47\}$&
$\{0,\frac{19t+19}{2},32t+29,51t+46\}$&$\{0,12t+9,25t+23,56t+47\}$\\
$\{0,15t+13,27t+23,60t+53\}$&$\{0,4t,18t+13,23t+20\}$&
$\{0,20t+19,\frac{43t+35}{2},50t+41\}$\\$\{0,23t+19,30t+24,34t+28\}$&
$\{0,3t,\frac{31t+27}{2},34t+29\}$ &\\\hline
\end{tabular}
\end{center}
\end{table}

\setcounter{table}{1}
\begin{table}[t]  \renewcommand\arraystretch{1.1}
\caption{ (Cont.) The remaining $18+x$ base blocks in Lemma \ref{thm:CDF13mod72}}\label{tab1-2}
\begin{center}\tabcolsep 0.03in
\begin{tabular}{|lll|}\hline

$x=5$ and $t$ is even &&\\\hline
$\{0,1,\frac{19t}{2}+10,39t+32\}$&$\{0,\frac{3t}{2}+1,24t+20,64t+51\}$&
$\{0,19t+20,22t+20,37t+34\}$\\$\{0,\frac{7t}{2}+1,\frac{31t}{2}+13,\frac{55t}{2}+23\}$&
$\{0,4t+4,15t+11,52t+44\}$&$\{0,13t+14,17t+14,44t+38\}$\\
$\{0,27t+23,30t+22,35t+29\}$&$\{0,5t+6,23t+19,51t+43\}$&
$\{0,6t+7,33t+28,65t+55\}$\\$\{0,\frac{35t}{2}+13,19t+15,\frac{135t}{2}+55\}$&
$\{0,11t+8,\frac{71t}{2}+29,47t+37\}$&$\{0,12t+9,\frac{61t}{2}+24,\frac{111t}{2}+47\}$\\
$\{0,12t+14,16t+15,42t+38\}$&$\{0,4t+5,8t+11,68t+59\}$&
$\{0,15t+12,23t+20,56t+47\}$\\$\{0,16t+13,30t+25,41t+34\}$&
$\{0,\frac{15t}{2}+5,20t+18,\frac{43t}{2}+18\}$&$\{0,18t+15,22t+18,41t+36\}$\\
$\{0,19t+16,38t+33,50t+44\}$&$\{0,3t+1,8t+9,11t+6\}$&
$\{0,21t+15,21t+17,57t+48\}$\\$\{0,25t+22,31t+30,43t+38\}$&
$\{0,4t+7,31t+29,69t+63\}$&\\\hline
\end{tabular}
\end{center}
\end{table}

\begin{Theorem}\label{thm:CDF}
There exists a $(v,4,1)$-CDF if and only if $v\equiv 1\pmod{12}$ and $v\neq 25$.
\end{Theorem}

\begin{proof} A $(v,4,1)$-CDF contains $(v-1)/12$ base blocks, so $v\equiv 1\pmod{12}$. For the sufficiency, it is known that a $(25,4,1)$-CDF does not exist by \cite{cm}. Combine the results of Lemma \ref{thm:CDF1mod72} and Lemma \ref{thm:CDF13mod72} to complete the proof.
\end{proof}

Next using similar techniques to those in the proof of Lemma \ref{thm:CDF13mod72}, we establish the existence of a $(v,4,4,1)$-CDF for any $v\equiv 4\pmod{12}$ and $v\not\in\{16,28\}$.

\begin{Theorem}\label{thm:CDF4mod12}
There exists a $(v,4,4,1)$-CDF if and only if $v\equiv 4\pmod{12}$ except for the two definite exceptions of $v=16$ and $28$.
\end{Theorem}

\proof A $(v,4,4,1)$-CDF contains $(v-4)/12$ base blocks, so $v\equiv 4\pmod{12}$. For the sufficiency, it is known that a $(v,4,4,1)$-CDF with $v\in\{16,28\}$ does not exist by \cite{cm}. For $v\equiv 4\pmod{12}$, $v\not\in\{16,28\}$ and $v\leq 208$, a $(v,4,4,1)$-CDF exists by \cite{cw}. For $v>208$, let $v=72t+12x+4$ where $t>2$ and $0\leq x\leq 5$. A $(v,4,4,1)$-CDF, $\F$, contains $6t+x$ base blocks. The first $6t-18$ base blocks are listed below:
\begin{center}
\begin{tabular}{llll}
$\{0$, & $43t+a_1+i$, & $31t+a_2+2i$, & $8t+a_3+3i\}$,\\
$\{0$, & $23t+b_1+i$, & $5t+b_2+2i$, & $8t+b_3+3i\}$,\\
$\{0$, & $41t+c_1+i$, & $25t+c_2+2i$, & $8t+c_3+3i\}$,\\
$\{0$, & $35t+d_1+i$, & $5t+d_2+2i$, & $d_3+3i\}$,\\
$\{0$, & $47t+e_1+i$, & $19t+e_2+2i$, & $e_3+3i\}$,\\
$\{0$, & $21t+f_1+i$, & $13t+f_2+2i$, & $f_3+3i\}$,\\
\end{tabular}
\end{center}
where $1\leq i\leq t-2$, $i\neq \lfloor t/2\rfloor$, and $a_j,b_j,c_j,d_j,e_j,f_j$ for $1\leq j\leq 3$ are given in the following table:
\begin{center}
\begin{tabular}{|l|lll|lll|lll|lll|lll|lll|}\hline
$x$ & $a_1$ & $a_2$ & $a_3$ & $b_1$ & $b_2$ & $b_3$ & $c_1$ & $c_2$ & $c_3$ & $d_1$ & $d_2$ & $d_3$ & $e_1$ & $e_2$ & $e_3$ & $f_1$ & $f_2$ & $f_3$\\\hline
$0$ & $2$ & $2$ & $2$ & $2$ & $1$ & $1$ & $1$ & $2$ & $0$ & $0$ & $0$ & $1$ & $3$ & $2$ & $2$ & $1$ & $0$ & $0$ \\
$1$ & $10$ & $7$ & $5$ & $6$ & $3$ & $4$ & $9$ & $6$ & $3$ & $7$ & $2$ & $1$ & $11$ & $5$ & $2$ & $4$ & $1$ & $0$\\
$2$ & $18$ & $14$ & $8$ & $10$ & $4$ & $7$ & $17$ & $11$ & $6$ & $12$ & $3$ & $1$ & $18$ & $9$ & $2$ & $8$ & $6$ & $0$\\
$3$ & $26$ & $20$ & $10$ & $14$ & $5$ & $8$ & $26$ & $18$ & $9$ & $18$ & $4$ & $1$ & $25$ & $11$ & $2$ & $10$ & $7$ & $0$\\
$4$ & $32$ & $24$ & $10$ & $18$ & $5$ & $8$ & $30$ & $20$ & $9$ & $24$ & $4$ & $1$ & $35$ & $17$ & $2$ & $16$ & $11$ & $0$\\
$5$ & $41$ & $31$ & $11$ & $21$ & $6$ & $9$ & $38$ & $25$ & $10$ & $30$ & $5$ & $1$ & $42$ & $18$ & $2$ & $17$ & $13$ & $0$
\\\hline
 \end{tabular} .
\end{center}
The remaining $18+x$ base blocks are provided in Table \ref{tab2} according to the parity of $t$. Then one can check that $\Delta \F=\Z_v\setminus\{0,v/4,v/2,3v/4\}$, and so $\F$ is a $(v,4,4,1)$-CDF. \qed

\begin{table}[htbp] \renewcommand\arraystretch{1.1}
\caption{The remaining $18+x$ base blocks in Theorem \ref{thm:CDF4mod12}}\label{tab2}
\begin{center}
\begin{tabular}{|lll|}\hline
$x=0$ and $t$ is odd &&\\\hline
$\{0,1,36t+1,67t+4\}$&$\{0,2,15t,45t+4\}$&
$\{0,\frac{3t-3}{2},14t-1,\frac{83t+1}{2}\}$\\$\{0,\frac{3t-1}{2},11t-2,\frac{95t+5}{2}\}$&
$\{0,\frac{3t+1}{2},13t+1,20t+1\}$&$\{0,3t-3,16t,19t\}$\\
$\{0,3t-2,35t-1,42t+1\}$&$\{0,3t-1,42t+3,53t+3\}$&
$\{0,6t-1,\frac{31t-1}{2},65t+3\}$\\$\{0,7t-2,21t,44t+2\}$&
$\{0,7t-1,11t-1,34t\}$&$\{0,\frac{15t+3}{2},25t+3,\frac{137t+9}{2}\}$\\
$\{0,\frac{19t-1}{2},\frac{43t+1}{2},\frac{111t+3}{2}\}$&$\{0,12t,16t-1,64t+2\}$&
$\{0,16t+1,22t+1,64t+3\}$\\$\{0,16t+2,21t+1,49t+4\}$&
$\{0,19t+3,30t,55t+2\}$&$\{0,25t+1,\frac{59t+1}{2},53t+2\}$\\\hline
$x=0$ and $t$ is even &&\\\hline
$\{0,1,36t+5,47t+2\}$&$\{0,\frac{3t}{2}+1,30t+3,64t+3\}$&
$\{0,3t-3,3t-1,\frac{9t}{2}-1\}$\\$\{0,3t-2,40t+2,56t+2\}$&
$\{0,\frac{7t}{2},\frac{19t}{2},22t\}$&$\{0,4t-1,23t,61t+5\}$\\
$\{0,5t-1,33t+1,50t+3\}$&$\{0,5t,32t+1,65t+3\}$&
$\{0,5t+1,21t,47t+3\}$\\$\{0,6t+1,13t+3,65t+5\}$&
$\{0,7t-2,11t-2,35t-1\}$&$\{0,7t,\frac{47t}{2}+2,\frac{99t}{2}+4\}$\\
$\{0,\frac{15t}{2}+1,\frac{61t}{2}+3,58t+4\}$&$\{0,8t+2,31t+3,61t+4\}$&
$\{0,\frac{19t}{2}+1,21t+1,57t+4\}$\\$\{0,11t+1,24t+2,27t+2\}$&
$\{0,14t+1,30t+4,42t+4\}$&$\{0,\frac{31t}{2}-1,33t,52t+3\}$\\\hline
$x=1$ and $t$ is odd &&\\\hline
$\{0,1,8t+4,23t+3\}$&$\{0,\frac{3t-1}{2},25t+5,30t+6\}$&$\{0,\frac{3t+1}{2},20t+4,57t+13\}$\\
$\{0,3t-3,3t-1,64t+11\}$&$\{0,\frac{7t+1}{2},21t+4,\frac{125t+25}{2}\}$&$\{0,4t,29t+6,65t+16\}$\\
$\{0,4t+1,7t+1,23t+5\}$&$\{0,4t+2,31t+7,57t+12\}$&$\{0,\frac{9t+3}{2},14t+3,27t+4\}$\\
$\{0,5t+3,23t+6,49t+12\}$&$\{0,6t+2,30t+8,45t+10\}$&$\{0,7t+3,40t+10,51t+13\}$\\
$\{0,7t+4,39t+11,42t+9\}$&$\{0,\frac{15t+7}{2},20t+5,\frac{43t+7}{2}\}$&$\{0,11t+1,16t+3,35t+6\}$\\
$\{0,\frac{23t+7}{2},\frac{89t+19}{2},\frac{95t+21}{2}\}$&$\{0,13t+2,30t+5,47t+9\}$&$\{0,19t+5,34t+6,41t+8\}$\\
$\{0,27t+7,\frac{73t+19}{2},\frac{85t+21}{2}\}$ &&\\\hline
$x=1$ and $t$ is even &&\\\hline
$\{0,2,49t+13,57t+17\}$&$\{0,\frac{3t}{2}+2,19t+5,\frac{135t}{2}+15\}$&$\{0,3t-3,11t,\frac{25t}{2}+1\}$\\
$\{0,3t-1,14t+1,59t+12\}$&$\{0,3t,48t+10,55t+12\}$&$\{0,\frac{7t}{2}+1,5t+1,47t+11\}$\\
$\{0,4t,\frac{23t}{2}+3,34t+5\}$&$\{0,5t+2,36t+10,57t+13\}$&$\{0,5t+3,16t+4,49t+12\}$\\
$\{0,6t+2,33t+6,\frac{85t}{2}+11\}$&$\{0,7t+4,39t+11,42t+9\}$&$\{0,8t+5,30t+8,49t+14\}$\\
$\{0,11t+3,24t+5,45t+9\}$&$\{0,13t+1,17t+3,36t+7\}$&$\{0,14t+2,29t+6,47t+9\}$\\
$\{0,15t+2,19t+3,26t+6\}$&$\{0,\frac{31t}{2}+3,25t+6,\frac{111t}{2}+13\}$
&$\{0,\frac{49t}{2}+5,\frac{55t}{2}+6,46t+9\}$\\
$\{0,28t+6,35t+6,35t+7\}$ &&\\\hline
$x=2$ and $t$ is odd &&\\\hline
$\{0,1,\frac{3t-1}{2},64t+21\}$&$\{0,2,3t+3,15t+7\}$&$\{0,3t-1,35t+11,54t+20\}$\\
$\{0,3t+2,31t+11,55t+21\}$&$\{0,\frac{7t+5}{2},\frac{19t+11}{2},\frac{85t+37}{2}\}$
&$\{0,\frac{9t+5}{2},\frac{15t+5}{2},\frac{99t+43}{2}\}$\\
$\{0,5t+2,23t+8,42t+18\}$&$\{0,5t+4,49t+22,52t+20\}$&$\{0,6t+2,14t+4,\frac{101t+41}{2}\}$\\
$\{0,7t+3,14t+5,50t+21\}$&$\{0,8t+4,\frac{19t+9}{2},57t+22\}$&$\{0,8t+5,25t+10,41t+16\}$\\
$\{0,8t+6,23t+9,44t+17\}$&$\{0,11t+5,19t+8,24t+11\}$&$\{0,\frac{23t+9}{2},\frac{31t+13}{2},\frac{83t+33}{2}\}$\\
$\{0,12t+5,38t+16,\frac{111t+45}{2}\}$&$\{0,12t+7,15t+4,47t+17\}$&$\{0,\frac{25t+13}{2},\frac{47t+19}{2},\frac{107t+41}{2}\}$\\
$\{0,17t+6,36t+13,48t+19\}$&$\{0,23t+7,27t+10,34t+11\}$&\\\hline
\end{tabular}
\end{center}
\end{table}

\setcounter{table}{2}
\begin{table}[htbp] \renewcommand\arraystretch{1.1}
\caption{ (Cont.) The remaining $18+x$ base blocks in Theorem \ref{thm:CDF4mod12}}\label{tab2-1}
\begin{center}\tabcolsep 0.03in
\begin{tabular}{|lll|}\hline
$x=2$ and $t$ is even &&\\\hline
$\{0,1,\frac{3t}{2}+1,64t+21\}$&$\{0,2,3t-1,15t+6\}$&$\{0,\frac{3t}{2}+2,31t+11,64t+23\}$\\
$\{0,3t-2,41t+16,64t+22\}$&$\{0,3t,36t+11,67t+25\}$&$\{0,3t+1,\frac{43t}{2}+8,45t+18\}$\\
$\{0,3t+2,17t+5,54t+22\}$&$\{0,3t+3,26t+12,42t+18\}$&$\{0,\frac{9t}{2}+2,\frac{19t}{2}+6,\frac{113t}{2}+22\}$\\
$\{0,6t+4,\frac{57t}{2}+10,65t+26\}$&$\{0,7t+1,35t+10,53t+18\}$&$\{0,\frac{15t}{2}+2,11t+5,68t+26\}$\\
$\{0,8t+2,44t+18,50t+21\}$&$\{0,8t+4,25t+11,48t+18\}$&$\{0,12t+4,27t+9,67t+26\}$\\
$\{0,12t+5,31t+13,61t+25\}$&$\{0,12t+6,19t+9,48t+19\}$&$\{0,\frac{25t}{2}+6,37t+16,57t+25\}$\\
$\{0,14t+6,\frac{61t}{2}+11,48t+17\}$&$\{0,21t+7,25t+10,44t+17\}$&\\\hline
$x=3$ and $t$ is odd &&\\\hline
$\{0,1,25t+18,69t+43\}$&$\{0,\frac{3t-3}{2},\frac{19t+13}{2},\frac{19t+17}{2}\}$&$\{0,\frac{3t-1}{2},31t+14,\frac{71t+35}{2}\}$\\
$\{0,3t,\frac{61t+29}{2},55t+30\}$&$\{0,3t+1,39t+22,\frac{113t+63}{2}\}$&$\{0,\frac{7t+5}{2},15t+9,\frac{141t+79}{2}\}$\\
$\{0,4t+3,11t+7,33t+17\}$&$\{0,4t+4,36t+22,39t+21\}$&$\{0,5t+4,8t+7,24t+16\}$\\
$\{0,5t+5,25t+16,37t+24\}$&$\{0,7t+2,31t+15,56t+30\}$&$\{0,\frac{15t+7}{2},\frac{43t+19}{2},31t+17\}$\\
$\{0,8t+4,19t+9,22t+11\}$&$\{0,8t+5,12t+7,57t+32\}$&$\{0,8t+9,23t+14,49t+30\}$\\
$\{0,12t+6,31t+16,55t+31\}$&$\{0,\frac{25t+15}{2},35t+18,41t+22\}$&$\{0,13t+7,28t+14,64t+37\}$\\
$\{0,15t+6,21t+9,34t+17\}$&$\{0,16t+8,24t+14,46t+23\}$&$\{0,23t+11,30t+14,41t+20\}$\\\hline
$x=3$ and $t$ is even &&\\\hline
$\{0,1,8t+10,11t+7\}$&$\{0,2,37t+24,68t+38\}$&$\{0,\frac{3t}{2},\frac{45t}{2}+10,\frac{61t}{2}+14\}$\\
$\{0,\frac{3t}{2}+2,\frac{19t}{2}+10,\frac{113t}{2}+32\}$&$\{0,3t-2,15t+6,46t+23\}$&$\{0,3t-1,19t+9,22t+11\}$\\
$\{0,3t+1,16t+8,34t+17\}$&$\{0,3t+3,23t+13,31t+18\}$&$\{0,\frac{7t}{2}+3,33t+17,66t+35\}$\\
$\{0,\frac{9t}{2}+3,\frac{15t}{2}+3,\frac{23t}{2}+6\}$&$\{0,5t+3,36t+23,58t+33\}$&$\{0,5t+4,25t+15,49t+30\}$\\
$\{0,6t+4,\frac{55t}{2}+14,37t+23\}$&$\{0,7t+2,15t+9,30t+14\}$&$\{0,7t+4,15t+7,39t+21\}$\\
$\{0,8t+6,23t+14,\frac{49t}{2}+15\}$&$\{0,12t+7,24t+13,50t+31\}$&$\{0,13t+8,\frac{73t}{2}+22,49t+29\}$\\
$\{0,16t+9,30t+15,47t+24\}$&$\{0,\frac{35t}{2}+9,\frac{57t}{2}+14,\frac{107t}{2}+31\}$&$\{0,19t+11,24t+16,51t+31\}$\\\hline
$x=4$ and $t$ is odd &&\\\hline
$\{0,\frac{3t-3}{2},\frac{19t+15}{2},\frac{101t+73}{2}\}$&$\{0,\frac{3t-1}{2},11t+6,\frac{119t+81}{2}\}$
&$\{0,\frac{3t+1}{2},26t+18,\frac{107t+73}{2}\}$\\
$\{0,3t-1,27t+17,54t+36\}$&$\{0,3t+1,27t+18,39t+29\}$&$\{0,3t+2,26t+19,\frac{135t+97}{2}\}$\\
$\{0,3t+3,8t+8,39t+30\}$&$\{0,5t+3,18t+14,47t+35\}$&$\{0,5t+4,21t+16,28t+19\}$\\
$\{0,6t+3,\frac{45t+29}{2},\frac{85t+63}{2}\}$&$\{0,6t+4,41t+28,55t+38\}$&$\{0,7t+5,11t+8,23t+16\}$\\
$\{0,8t+6,12t+10,15t+10\}$&$\{0,8t+7,38t+28,42t+30\}$&$\{0,\frac{19t+17}{2},\frac{35t+27}{2},\frac{113t+83}{2}\}$\\
$\{0,11t+5,19t+15,54t+37\}$&$\{0,17t+12,17t+13,36t+29\}$&$\{0,\frac{23t+17}{2},15t+11,40t+29\}$\\
$\{0,12t+9,40t+30,47t+32\}$&$\{0,12t+12,15t+9,65t+46\}$&$\{0,13t+12,16t+10,55t+41\}$\\
$\{0,11t+7,11t+9,39t+27\}$&&\\\hline
$x=4$ and $t$ is even &&\\\hline
$\{0,1,\frac{3t}{2}+2,64t+44\}$&$\{0,2,19t+18,69t+55\}$&$\{0,\frac{3t}{2},5t+3,\frac{33t}{2}+11\}$\\
$\{0,3t-2,11t+8,15t+10\}$&$\{0,29t+21,32t+24,36t+27\}$&$\{0,\frac{9t}{2}+3,27t+17,32t+21\}$\\
$\{0,6t+4,23t+18,44t+34\}$&$\{0,20t+17,31t+22,55t+39\}$&$\{0,12t+11,39t+30,56t+42\}$\\
$\{0,\frac{19t}{2}+8,31t+24,58t+42\}$&$\{0,\frac{31t}{2}+10,23t+15,\frac{83t}{2}+30\}$&
$\{0,12t+8,25t+19,69t+51\}$\\
$\{0,8t+5,12t+9,47t+32\}$&$\{0,13t+12,19t+17,30t+23\}$&$\{0,14t+11,30t+22,38t+29\}$\\
$\{0,\frac{19t}{2}+9,\frac{25t}{2}+11,\frac{85t}{2}+32\}$&$\{0,\frac{35t}{2}+13,\frac{49t}{2}+17,61t+45\}$&$\{0,18t+14,25t+17,54t+37\}$\\
$\{0,7t+2,25t+18,46t+33\}$&$\{0,23t+14,28t+19,35t+24\}$&$\{0,23t+17,31t+23,42t+32\}$\\
$\{0,3t,23t+16,39t+28\}$&&\\\hline
\end{tabular}
\end{center}
\end{table}

\setcounter{table}{2}
\begin{table}[htbp] \renewcommand\arraystretch{1.1}
\caption{ (Cont.) The remaining $18+x$ base blocks in Theorem \ref{thm:CDF4mod12}}\label{tab2-2}
\begin{center}\tabcolsep 0.03in
\begin{tabular}{|lll|}\hline
$x=5$ and $t$ is odd &&\\\hline
$\{0,1,3t-2,36t+30\}$&$\{0,2,24t+23,64t+58\}$&$\{0,\frac{3t-3}{2},29t+23,64t+53\}$\\
$\{0,\frac{3t-1}{2},3t,\frac{43t+33}{2}\}$&$\{0,11t+10,16t+14,33t+29\}$&$\{0,4t+2,\frac{33t+31}{2},\frac{137t+123}{2}\}$\\
$\{0,4t+3,7t+4,58t+52\}$&$\{0,4t+5,16t+16,61t+58\}$&$\{0,5t+5,\frac{45t+41}{2},27t+25\}$\\
$\{0,6t+4,31t+27,60t+51\}$&$\{0,15t+13,31t+26,61t+53\}$&$\{0,22t+21,30t+25,43t+38\}$\\
$\{0,\frac{19t+15}{2},\frac{57t+47}{2},\frac{83t+75}{2}\}$&$\{0,11t+7,\frac{71t+59}{2},47t+40\}$&$\{0,11t+8,23t+20,64t+54\}$\\
$\{0,3t-1,8t+5,53t+45\}$&$\{0,12t+14,19t+17,50t+48\}$&$\{0,14t+13,38t+33,41t+36\}$\\
$\{0,6t+5,\frac{59t+51}{2},\frac{125t+111}{2}\}$&$\{0,15t+14,23t+21,39t+36\}$&$\{0,22t+17,30t+26,47t+42\}$\\
$\{0,7t+5,11t+9,44t+40\}$&$\{0,26t+23,29t+25,41t+35\}$&\\\hline
$x=5$ and $t$ is even &&\\\hline
$\{0,2,39t+36,60t+52\}$&$\{0,19t+19,22t+16,22t+17\}$&$\{0,\frac{3t}{2}+1,\frac{19t}{2}+9,25t+22\}$\\
$\{0,3t,\frac{15t}{2}+4,11t+7\}$&$\{0,3t+1,14t+12,22t+18\}$&$\{0,3t+2,\frac{61t}{2}+26,55t+48\}$\\
$\{0,3t+3,14t+13,50t+44\}$&$\{0,4t+4,11t+8,19t+18\}$&$\{0,5t+4,43t+39,49t+44\}$\\
$\{0,5t+5,27t+24,31t+29\}$&$\{0,5t+6,20t+17,32t+28\}$&$\{0,6t+6,47t+40,55t+49\}$\\
$\{0,7t+3,29t+24,36t+29\}$&$\{0,12t+13,15t+12,48t+43\}$&$\{0,8t+11,16t+16,32t+31\}$\\
$\{0,11t+6,41t+33,54t+47\}$&$\{0,11t+9,37t+34,53t+48\}$&$\{0,\frac{23t}{2}+10,30t+26,\frac{119t}{2}+51\}$\\
$\{0,8t+4,21t+17,39t+32\}$&$\{0,15t+13,\frac{33t}{2}+15,39t+35\}$&$\{0,\frac{35t}{2}+15,\frac{43t}{2}+17,\frac{125t}{2}+54\}$\\
$\{0,\frac{3t}{2},\frac{57t}{2}+23,38t+34\}$&$\{0,25t+25,37t+35,41t+38\}$&\\\hline
\end{tabular}
\end{center}
\end{table}

\section{Summary and proofs of the main theorems}

Now we are to prove Theorems \ref{thm:main} and \ref{thm:main-OOC}.

\vspace{0.38cm}
\noindent {\bf Proof of Theorem \ref{thm:main}:} A cyclic $(v,4,1)$-design exists only if $v\equiv 1,4\pmod{12}$ and $v\not\in \{16,25,28\}$ by \cite{cm}. For the sufficiency, it follows from Lemma \ref{lem:CDF-CD}(1) and Theorem \ref{thm:CDF} that a cyclic $(v,4,1)$-design exists for any $v\equiv 1 \pmod{12}$ and $v\not\neq 25$. By Lemma \ref{lem:CDF-CD}(2) and Theorem \ref{thm:CDF4mod12}, a cyclic $(v,4,1)$-design exists for any $v\equiv 4 \pmod{12}$ and $v\not\in\{16,28\}$. \qed

\vspace{0.38cm}
\noindent {\bf Proof of Theorem \ref{thm:main-OOC}:} By Lemma \ref{lem:CDP-OOC}, a $(v,4,1)$-CDP with $b$ base blocks is equivalent to a $(v,4,1)$-OOC with $b$ codewords. A $(v,k,1)$-CDF contains $(v-1)/k(k-1)$ base blocks, and a $(v,k,k,1)$-CDF contains $(v-k)/k(k-1)$ base blocks. An optimal $(v,4,1)$-OOC contains $\lfloor (v-1)/12\rfloor$ codewords. It follows from Theorem \ref{thm:main} that Theorem \ref{thm:main-OOC} holds. \qed

\vspace{0.38cm}

Many open problems were raised by Reid and Rosa \cite[Section 14]{rr} in their survey on $(v,4,1)$-designs, where the first two problems are to find a direct proof of the existence of $(v,4,1)$-designs and to show that a cyclic $(v,4,1)$-design exists for all admissible $v\geq 37$, respectively. This paper gives a solution to both of the problems.

A further research topic is to examine the existence of a $(v,4,1)$-CDP with $\lfloor (v-1)/12\rfloor$ base blocks for any $v\not\equiv 1,4\pmod{12}$, which yields an optimal $(v,4,1)$-OOC.
There should be no serious obstacle to extending our technique in Section 2 to construct them, but much more time will be required to search for the sporadic base blocks. Also we believe that our technique can be employed to construct 1-rotational $(v,4,1)$-designs (see \cite{rr} for the definition), which can be seen as a special kind of $(v,4,1)$-CDPs.

A more interesting topic is to examine the existence of cyclic $(v,k,1)$-designs with $k\geq 5$. A possible approach is to require that each base block is of the form $$\{0,\  \alpha_{1,1}t+\alpha_{1,2}+i,\ \alpha_{2,1}t+\alpha_{2,2}+2i,\ \ldots,\ \alpha_{k-1,1}t+
\alpha_{k-1,2}+(k-1)i\}$$
for some $t$ that is related to $v$. We leave it as an open problem in the more challenging case.

\subsection*{Acknowledgements}

The authors thank the editor, Marco Buratti, and the anonymous referees for their valuable comments and suggestions that helped improve the equality of the paper.


\begin{thebibliography}{99}

\bibitem{AbelBuratti2004}
R.J.R.~Abel and M.~Buratti, Some progress on $(v,4,1)$ difference families and optical orthogonal codes, J. Combin. Theory A, 106 (2004), 59--75.

\bibitem{abel}
R.J.R.~Abel and M.~Buratti, Difference families, in: C.J.~Colbourn, J.H.~Dinitz (Eds.), Handbook of Combinatorial Designs (2nd Edition), CRC Press, Boca Raton, FL, 2007, 392--410.

\bibitem{Bose}
R.C.~Bose, On the construction of balanced incomplete block designs, Ann. Eugenics, 9 (1939), 353--399.


\bibitem{Buratti95}
M.~Buratti, Improving two theorems of Bose on difference families, J. Combin. Des., 3 (1995), 15--24.

\bibitem{Buratti95-1}
M.~Buratti, Constructions for $(q,k,1)$ difference families with $q$ a prime power and $k=4,5$,
Discrete Math., 138 (1995), 169--175.

\bibitem{Buratti97}
M.~Buratti, From a $(G,k,1)$ to a $(C_k\oplus G,k,1)$ difference family, Des. Codes Cryptogr., 11 (1997), 5--9.

\bibitem{Buratti98}
M.~Buratti, Recursive constructions for difference matrices and relative difference families, J. Combin.
Des., 6 (1998), 165--182.

\bibitem{Buratti02}
M.~Buratti, Cyclic designs with block size 4 and related optimal optical orthogonal codes, Des. Codes Cryptogr., 26 (2002), 111--125.

\bibitem{bp2}
M.~Buratti and A.~Pasotti, Further progress on difference families with block size $4$ or $5$, Des. Codes Cryptogr., 56 (2010), 1--20.

\bibitem{c}
Y.~Chang, Some cyclic BIBDs with block size four, J. Combin. Des., 12 (2004), 177--183.

\bibitem{cc}
P.L.~Check and C.J.~Colbourn, Concerning difference families with block size four, Discrete Math., 133 (1994), 285--289.

\bibitem{cw}
K.~Chen and R.~Wei, A few more cyclic Steiner 2-designs, Electron. J. Combin., 13 (2006), \#R10, 1--13.

\bibitem{chen45}
K.~Chen and L.~Zhu, Existence of $(q,k,1)$ difference families with $q$ a prime power and $k=4,5$, J. Combin. Des., 7 (1999), 21--30.

\bibitem{csw}
F.R.K.~Chung, J.A.~Salehi, and V.K.~Wei, Optical orthogonal codes: design, analysis and applications, IEEE Trans. Inform. Theory, 35 (1989), 595--604.

\bibitem{cc84}
M.J.~Colbourn and C.J.~Colbourn, Recursive constructions for cyclic block designs, J. Statist. Plann.  Inference, 10 (1984), 97--103.

\bibitem{cm}
M.J.~Colbourn and R.A.~Mathon, On cyclic Steiner 2-designs, Ann. Discrete Math., 7 (1980), 215--253.

\bibitem{gsm}
G.~Ge, Y.~Miao, and X.~Sun, Perfect difference families, perfect difference matrices, and related combinatorial structures, J. Combin. Des., 18 (2010), 415--449.

\bibitem{Hanani}
H.~Hanani, The existence and construction of balanced incomplete block designs, Ann. Math. Stat., 32 (1961), 361--386.

\bibitem{Heffter1896}
L.~Heffter, \"Uber Nachbarconfigurationen, Tripelsysteme und metacyklische Gruppen, Deutsche Mathem. Vereinig. Jahresber., 5 (1896), 67--69.

\bibitem{Heffter1897}
L.~Heffter, \"Uber Tripelsysteme, Math. Ann., 49 (1897), 101--112.

\bibitem{Kirkman}
T.P.~Kirkman, On a problem in combinations, Cambridge and Dublin Math. J., 2 (1847) 191--204.

\bibitem{k}
E.~K\"{o}hler, $k$-difference-cycles and the construction of cyclic $t$-designs, in: Geometries and Groups, in: Lecture Notes in Math., Springer-Verlag, Berlin, 893 (1981), 195--203.

\bibitem{Netto}
E.~Netto, Zur Theorie der Tripelsysteme, Math. Ann., 42 (1893), 143--152.

\bibitem{Peltesohn}
R.~Peltesohn, Eine L\"{o}sung der beiden Heffterschen Differenzenprobleme, Compositio
Math., 6 (1939), 251--257.

\bibitem{rr}
C.~Reid and A.~Rosa, Steiner systems $S(2,4,v)$ - a survey, Electron. J. Combin., (2010), \#DS18.

\bibitem{GAP4}
The GAP~Group, GAP -- Groups, Algorithms, and Programming, Version 4.11.1, 2021,   \url{https://www.gap-system.org}.


\bibitem{yl}
Y.~Yang and X.~Lin, Coding Theory and Cryptography (in Chinese), People’s Post and Telecommunications Press, Beijing, China, 1992.

\bibitem{Yin98}
J.~Yin, Some combinatorial constructions for optical orthogonal codes, Discrete Math., 185 (1998), 201--219.

\bibitem{zfw}
M.~Zhang, T.~Feng, and X.~Wang, The existence of cyclic $(v,4,1)$-designs, 2022, \url{https://doi.org/10.5281/zenodo.6370238}.

\end{thebibliography}
\end{document}